\documentclass{amsart}
\input{epsf.tex}
\usepackage{epsfig}
\usepackage{graphicx,color,graphics,verbatim,textcomp}
\newtheorem{theorem}{Theorem}[section]
\newtheorem{lemma}[theorem]{Lemma}
\newtheorem{property}[theorem]{Property}
\newtheorem{proposition}[theorem]{Proposition}

\theoremstyle{definition}

\newtheorem{question}[theorem]{Question}

\theoremstyle{remark}

\numberwithin{equation}{section}

\def\Mod{\mbox{\rm{Mod}}}
\def\ML{{\mathcal {ML}}}
\def\PML{{\mathbb P}{\mathcal {ML}}}

\newcommand{\Homeo}{{\text{Homeo}}}
\newcommand{\co}{\colon\thinspace}

\begin{document}

\title{A fake Schottky group in $\mathrm{\mathbf{Mod}}(S)$}

\author{Richard P. Kent IV}
\address{Department of Mathematics, Brown University, Providence, RI 02912}
\email{rkent@math.brown.edu}
\thanks{The first author was supported in part by an NSF postdoctoral fellowship.}

\author{Christopher J. Leininger}
\address{Department of Mathematics, University of Illinois, Urbana, IL 61801}
\email{clein@math.uiuc.edu}
\thanks{The second author was supported in part by NSF Grant DMS-0603881.}

\subjclass{Primary 20F65 ; Secondary 30F60, 57M07, 57M50}
\date{September 23, 2008}

\keywords{Schottky group, mapping class group}

\begin{abstract}
We use the classical construction of Schottky groups in hyperbolic geometry to produce non-Schottky subgroups of the mapping class group.
\end{abstract}

\maketitle
\markboth{RICHARD P. KENT IV AND CHRISTOPHER J. LEININGER}{A FAKE SCHOTTKY GROUP IN $\mathrm{Mod}(S)$}

\section{Introduction}
In hyperbolic geometry, a Schottky group is a free convex cocompact Kleinian group, classically constructed as follows.  
Pick four pairwise disjoint closed balls $B_1^-$, $B_2^-$, $B_1^+$, $B_2^+$  in $S^{n-1}_\infty$, the ideal boundary of hyperbolic $n$--space.  
Suppose there are isometries $f_1$ and $f_2$ so that 
\[ 
f_i(B_i^-) = \overline{S^{n-1}_\infty - B_i^+}. 
\]
Then $\langle f_1, f_2 \rangle$ is a Schottky group isomorphic to the free group $F_2$ of rank two.

Now let $S$ be a closed surface of genus $g \geq 2$ and let $\Mod(S) = \pi_0(\Homeo^+(S))$ be its mapping class group.  
By way of analogy with the theory of Kleinian groups, B. Farb and L. Mosher defined \cite{FM} a notion of convex cocompactness for subgroups of $\Mod(S)$.
In this setting,  a \textit{Schottky group} is a free convex cocompact subgroup of $\Mod(S)$.  In \cite{KL1,KL2}, we extended Farb and Mosher's analogy, providing several characterizations of convex cocompactness borrowed from the Kleinian setting (see also Hamenst\"adt \cite{H}).  
The analogy is an imperfect one, see \cite{KL3} and the references there, and we point out some new imperfections here.

\begin{theorem} \label{funny example1}
There exist pseudo-Anosov elements $f_1$ and $f_2$ in $\Mod(S)$ and pairwise disjoint closed balls $B_1^-, B_2^-,B_1^+, B_2^+$ in $\PML(S)$ for which
\[ 
f_i(B_i^-) = \overline{\PML(S) - B_i^+}
\]
and yet $\langle f_1, f_2 \rangle \cong F_2$ is not a Schottky group.
\end{theorem}
The construction is based on work of N. Ivanov, and it is clear from his work in \cite{I} that he was aware of this construction (see also McCarthy \cite{Mc}).
The group $G = \langle f_1,f_2 \rangle$ contains reducible elements and so fails to be convex cocompact.
It is worth noting that there are sufficiently high powers of the $f_i$ that generate a Schottky group, as proven by Farb and Mosher \cite{FM}, see also \cite{KL1,H}.

Part of the analogy between Kleinian groups and mapping class groups was developed by J. McCarthy and A. Papadopoulos \cite{MP}, who constructed a limit set $\Lambda_G$ and domain of discontinuity $\Delta_G \subset \PML(S) - \Lambda_G$ for any subgroup $G < \Mod(S)$, see Section \ref{propersection}. 
Unlike in the Kleinian setting,  $\Delta_G \neq \PML(S) - \Lambda_G$ in general.  
While examples illustrate the necessity of taking an open set strictly smaller than $\PML(S) - \Lambda_G$ as a domain of discontinuity, it is not clear that $\Delta_G$ is an optimal choice. 
In \cite{KL1}, we asked whether or not $\Delta_G$ is the largest open set on which $G$ acts properly discontinuously---see Question 3 there.
Here, we answer this in the negative.

There is an obvious open set on which our group $G = \langle f_1, f_2 \rangle$ acts properly discontinuously and cocompactly, namely
\[
\Omega = \bigcup_{g \in G} \ g \cdot \Theta 
\]
where $\Theta$ is the closure of the complement of our four balls.
To see that $\Omega$ is open, note that $\Theta$ is contained in the interior $\mathcal{U}$ of 
\[
f_1(\Theta) \cup f_1^{-1}(\Theta) \cup f_2(\Theta) \cup f_2^{-1}(\Theta)
\]
and that
\[
\Omega = \bigcup_{g \in G} \ g \cdot \mathcal{U}
\]
If the $f_i$ are chosen carefully, the set $\Omega$ will contain $\Delta_G$ properly, and we have the following theorem.
\begin{theorem} \label{pdexample}
There are irreducible subgroups $G < \Mod(S)$ for which $\Delta_G$ is \textnormal{not} the largest open set on which $G$ acts properly discontinuously.
\end{theorem}

Asymmetry of the construction provides another domain $\Omega'$ on which $G$ acts properly discontinuously, and we will also show that $G$ does not act properly discontinuously on the union $\Omega \cup \Omega'$.

Though $\Delta_G$ is not a maximal domain of discontinuity, we show in Section \ref{finalcomments} that, for the groups in Theorem \ref{pdexample}, it is nonetheless the intersection of all such maximal domains.

\section{Surface dynamics} \label{prelimsection}
If $\mathcal X$ is a subset of $\PML(S)$, we let
\[
Z \mathcal X= \{ [\nu] \in \PML(S) \, | \, i(\nu,\mu) = 0 \mbox{ for some } [\mu] \in \mathcal X \}
\]
be the \textit{zero locus of $\mathcal{X}$}.
If $\mathcal X =\{ [x] \}$ we sometimes write $Z x$ for $Z \mathcal X$.

If $f$ is pseudo-Anosov, then it acts with \textit{north--south} dynamics on $\PML(S)$, meaning
that it has unique attracting and repelling fixed points $[\mu^+_f]$ and $[\mu^-_f]$, respectively---all other points are attracted to $[\mu^+_f]$ under iteration of $f$. In fact, for any neighborhood $U$ of $[\mu^+_f]$ and any
compact set $K \subset \PML(S) - \{[\mu^-_f]\}$, there is a natural number $N$ so that
\begin{equation} \label{ping}
f^n(K) \subset U
\end{equation}
for any $n \geq N$.

Ivanov proves that there is a similar situation for most pure reducible elements (see the Appendix of \cite{I}).  
In particular, suppose $\alpha$ is a nonseparating simple closed curve in $S$ preserved by a mapping class $\phi$ that is pseudo-Anosov when restricted to $S - \alpha$.  
Let $[\mu^+_\phi]$ and $[\mu^-_\phi]$ be the stable and unstable laminations for $\phi$ in $S - \alpha$ considered laminations on $S$,
and note that 
\[
Z\mu^-_\phi = \big\{ [s\mu^-_\phi + (1-s)\alpha] \in \PML(S) \, | \, s \in [0,1] \big\}. 
\]
If $K \subset \PML(S) - Z\mu^-_\phi $ is a compact set and $U$ a neighborhood of $[\mu^+_\phi]$, then there is an $N > 0$ such that for all $n \geq N$ we have
\begin{equation} \label{pong}
\phi^n(K) \subset U.
\end{equation}

Given a mapping class $g$ of either type above,
let $\lambda(g)$ denote the \textit{expansion factor} of $g$, the number such that
\[
g(\mu_g^+) = \lambda(g) \mu_g^+.
\]

\section{The construction}

Let $\alpha$ be a nonseparating curve fixed by a mapping class $\phi$ that is pseudo-Anosov on $S - \alpha$, and let $[\mu^+_{\phi}]$ and $[\mu^-_{\phi}]$ be as in the previous section.  

Let $\mathbb{S}_\phi \subset \PML(S)$ be a bicollared $(6g-8)$--dimensional sphere  dividing $\PML(S)$ into two closed balls $\mathcal A_{\phi}$ and $\mathcal B_{\phi}$ containing $Z\mu^-_{\phi}$ and $[\mu^+_\phi]$, respectively.

According to \eqref{pong}, there is an $N > 0$ so that for all $n \geq N$ we have
\[ 
\phi^n(\mathcal B_{\phi}) \subset \mathrm{int}(\mathcal B_{\phi}). 
\]
So we choose an $n \geq N$, let $h = \phi^n$, $B^-_h = \mathcal A_{\phi}$, and $B^+_h = h(\mathcal B_{\phi})$.

Recall H.~Masur's theorem \cite{Ma} that the set
\[
\{([\mu^+_\psi],[\mu^-_\psi]) \, | \, \psi \in \Mod(S) \mbox{ pseudo-Anosov} \}
\]
 is dense in $\PML(S) \times \PML(S)$.
So we choose a pseudo-Anosov $\psi$ whose fixed points $[\mu^+_\psi]$ and $[\mu^-_\psi]$ lie in $\PML(S) - (B^-_h \cup B^+_h)$.
We let $\mathbb{S}_\psi \subset \PML(S)- (B^-_h \cup B^+_h)$ be a bicollared $(6g-8)$--sphere which bounds two balls:  $\mathcal A_\psi \subset \PML(S) - (B^-_h \cup B^+_h)$ containing $[\mu^-_\psi]$ and $\mathcal B_\psi$ containing $[\mu^+_\psi]$.  
As $\PML(S) - (B^-_h \cup B^+_h)$ is a neighborhood of $[\mu^+_\psi]$, \eqref{ping} provides an $M > 0$ such that for all $m \geq M$, we have $\psi^m(\mathcal B_\psi) \subset \PML(S) - (B^-_h \cup B^+_h)$.  
Arguing as in \cite{I}, we may choose $m$ so that $\psi^mh$ is pseudo-Anosov, and we do so.  
We let $f = \psi^m$, $B^-_f = \mathcal A_\psi$, and $B^+_f = f(\mathcal B_\psi)$.
\begin{figure}
\begin{center}
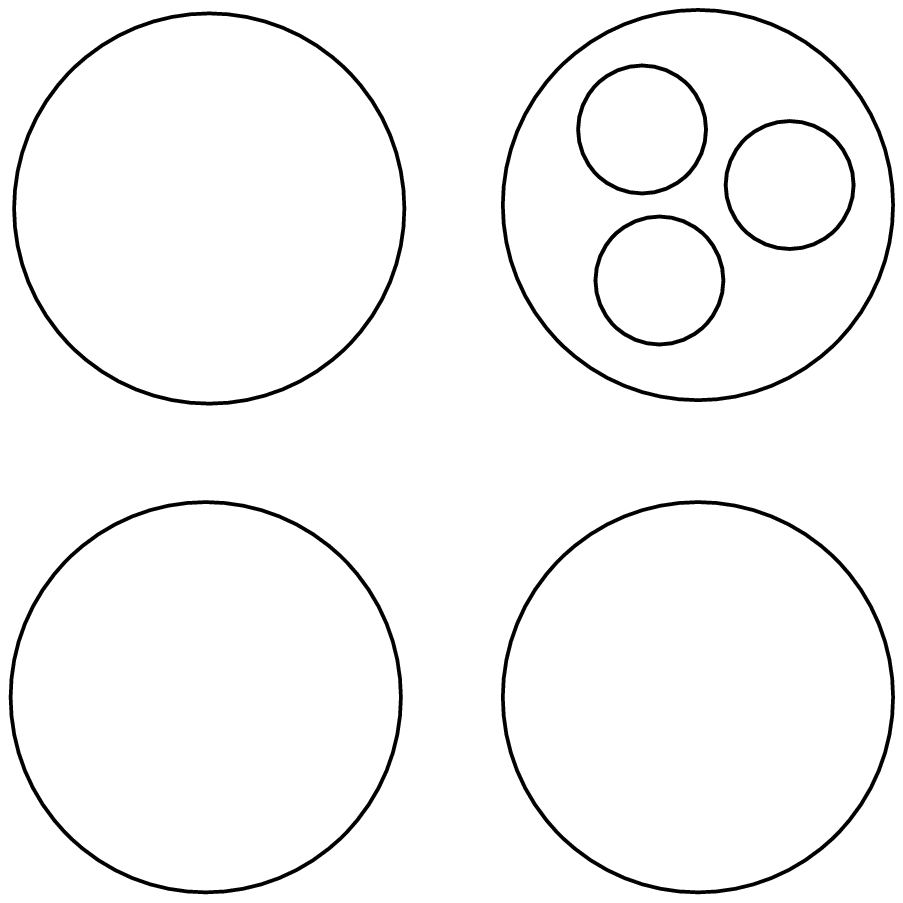
\end{center}
\caption{}\label{schottky1}
\end{figure}

We now have elements $f$, $h$, and pairwise disjoint closed balls
\[
B^-_h,B^+_h,B^-_f,B^+_f
\]
with
\[
h(B^-_h) = \overline{\PML(S) - B^+_h} \mbox{ and } f(B^-_f) = \overline{\PML(S) - B^+_f}. 
\]
See Figure \ref{schottky1}.

Let $G = \langle f,h \rangle$, set
\[ 
\Theta = \overline{\PML(S) - \left( B^-_h \cup B^+_h \cup B^-_f \cup B^+_f \right)}. 
\]
and let
\[\Omega = \bigcup_{g \in G} g \cdot \Theta.\]
The group $G$ acts on $\Omega$ properly discontinuously and cocompactly with fundamental domain $\Theta$, and the usual ping--pong argument implies that $G \cong F_2$.

A slight modification now provides the desired example.  

We let $f_1 = fh$ and $f_2 = f$, both pseudo-Anosov by construction.  Of course, $G = \langle f_1,f_2 \rangle$, and we need only find balls $B_1^\pm$ and $B_2^\pm$ with
\begin{equation}\label{nest}
f_i(B_i^-) = \overline{\PML(S) - B_i^+} .
\end{equation}
Set $B_1^- = B^-_h$ and $B_1^+ = f(B^+_h)$.  
The ball $B_2^-$ is constructed as a regular neighborhood of $B^-_f \cup B^+_h \cup \delta$ in $\PML(S) - (B^-_h \cup B^+_f)$, where $\delta$ is an arc in $\Theta$ from $B^-_f$ to $B^+_h$.
The ball $B_2^+$ is defined to be $\overline{\PML(S) - f(B_2^-)}$.  See Figure \ref{schottky2}.

\begin{figure}
\begin{center}
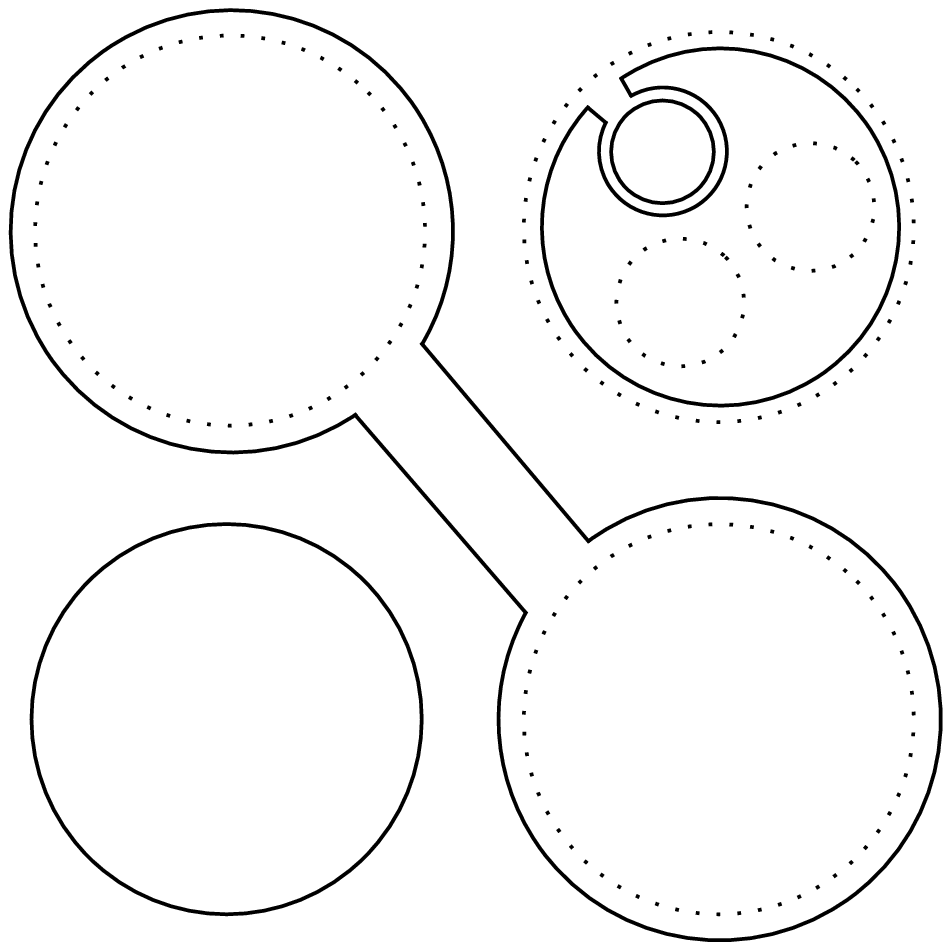
\end{center}
\caption{}\label{schottky2}
\end{figure}

One can now check \eqref{nest}.

\section{Proper discontinuity}\label{propersection}

Let $G = \langle h,f \rangle$ be the group constructed in the previous section, and let $\partial G$ be the Gromov boundary of $G$.
By the work in \cite{MP}, the limit set 
\[ 
\Lambda_G = \overline{ \{ [\mu^+_g] \in \PML(S) \, | \, g \in G \mbox{ is pseudo-Anosov} \}}
\]
is the unique minimal closed $G$--invariant subset of $\PML(S)$.  
In \cite{KL2} we showed that one may choose $h$ and $f$ as above so that $G$ has the following property.
\begin{property} \label{kluniform}
There exists a continuous $G$--equivariant homeomorphism
\[
\mathfrak{I}\co \partial G \to \Lambda_G. 
\]
Moreover, for each $x \in \partial G$ which is a fixed point of a conjugate $^gh$ of $h$, $\mathfrak{I}(x)$ is the stable or unstable lamination of that conjugate $^gh$ (respecting the dynamics). Otherwise $\mathfrak{I}(x)$ is a uniquely ergodic filling lamination.  In particular, every element $g \in G$ not conjugate to a power of $h$ is pseudo-Anosov.
\end{property}
We henceforth assume that $G$ satisfies Property \ref{kluniform}.

The domain of discontinuity is defined to be
\[ 
\Delta_G = \PML(S) - Z \Lambda_G .
\]
This is an open set on which $G$ acts properly discontinuously \cite{MP}, which justifies the name.

We may describe the zero locus $Z \Lambda_G$ for $G$ explicitly.
For each conjugate $^g h$ of $h$, we have the attracting and repelling fixed points $x^\pm_{^g\!h}$ in $\partial G$.  
By Property \ref{kluniform}, the map $\mathfrak{I}$ sends these to the stable and unstable laminations
\[ 
\mathfrak{I}(x^\pm_{^g\!h}) = [\mu^\pm_{^g\!h}] = g [\mu^\pm_h].
\]
For any such point $g [\mu^\pm_h] \in \Lambda_G$, the set $Z g\mu^\pm_h = gZ\mu^\pm_h$ is a $1$--simplex in $Z \Lambda_G$. 
Since $\mathfrak{I}(x)$ is uniquely ergodic and filling for every other point $x \in \partial G$, it follows that $Z \Lambda_G$ is the union of $\Lambda_G$ and all of these intervals.

The intervals $Z\mu^-_h$ and $Z\mu^+_h$ intersect each other at $\alpha$, and so the union 
\[
\mathbb{J}_h = Z\mu^-_h \cup Z\mu^+_h
\]
 is an interval joining $\mu^-_h$ to $\mu^+_h$.
All in all, we have
\begin{equation}\label{descriptionZ}
Z \Lambda_G = \Lambda_G \cup \bigcup_{g \in G}  g\, \mathbb{J}_h 
\end{equation}

We impose one further restriction on $h$ and $f$---more precisely, on the balls $B^\pm_f$.
Since the fixed points of $f$ do not meet the interval $\mathbb{J}_h$, we may replace $f$ with a power so that the balls $B^\pm_f$ are disjoint from this interval.  
This implies that
\[
\mathrm{int}\, Z \mu^+_h = Z \mu^+_h \cap \bigcup_{n\in \mathbb Z} h^n \Theta
\]
and so $Z \mu^+_h$ intersects the $h^n$ translates of $\Theta$, and no other $G$--translates.
As $Z \mu^-_h$ does not intersect $\Omega$, these are the only $G$--translates of $\Theta$ that $\mathbb{J}_h$ intersects.
Write $\Sigma^\pm_h = \partial B^\pm_h$ and $\Sigma^\pm_f= \partial B^\pm_f$.  

We claim that
\[
\Sigma^+_f \cap Z \Lambda_G = \emptyset.
\]
To see this, note that if $\Sigma^+_f$ nontrivially intersected $Z \Lambda_G$, it would do so in some $g \mathbb{J}_h$, by \eqref{descriptionZ}; and then $g$ must be a power of $h$, since $\Sigma^+_f$ lies in $\Theta$.
But  $h \mathbb{J}_h = \mathbb{J}_h$, and so  $\Sigma^+_f$ would intersect $\mathbb{J}_h$, contrary to our choice of $f$.  
The claim follows.

Now, Theorem \ref{pdexample} will follow from

\begin{theorem} \label{pdexample2}
The set $\Delta_G$ is properly contained in $\Omega$. In fact,
\[
\Omega = \PML(S) - \left( \Lambda_G \cup \bigcup_{g\in G} g Z\mu^-_h \right).
\]
\end{theorem}

First note that $\Delta_G \neq \Omega$ as 
 $\Sigma^-_h \subset \Theta \subset \Omega$ nontrivially intersects  $\mathbb{J}_h \subset Z \Lambda_G = \PML(S) - \Delta_G$.  

To prove the containment, we must gather some information about the complement of $\Omega$.  
Let $\mathfrak{X} = \PML(S) - \Omega$.

\begin{lemma}
There is a continuous $G$--equivariant map
\[
\mathfrak{K}\co\mathfrak{X} \to \partial G.
\]
\end{lemma}
\begin{proof}
The spheres $\Sigma^\pm_h$ and $\Sigma^\pm_f$ are bicollared with collars $N(\Sigma^\pm_h)$ and $N(\Sigma^\pm_f)$.
We assume as we may that
\[ 
h(N(\Sigma^-_h)) = N(\Sigma^+_h) \mbox{ and } f(N(\Sigma^-_f)) = N(\Sigma^+_f) 
\]
and that all of the $G$--translates of these collars are pairwise disjoint.

Let $\mathcal G$ be the Cayley graph of $G$ and identify $\partial G = \partial \mathcal G$.  
We define a continuous $G$--equivariant map
\[ 
\mathfrak{K}_0\co\Omega \to \mathcal G 
\]
by identifying $\mathcal G$ with the tree dual to the hypersurface
\[
\bigcup_{g \in G} g\big(\Sigma^-_h\big) \cup \bigcup_{g \in G} g\big(\Sigma^-_f \big)
\]
in $\Omega$ and projecting in the usual manner, see \cite{Sh}.

The map $\mathfrak{K}_0$ extends continuously to a $G$--equivariant map
\[
 \mathfrak{K}\co \PML(S) \to \overline{\mathcal G} = \mathcal G \cup \partial G
\] 
whose restriction to $\mathfrak{X}$ is the map we desire.
The extension is described concretely as follows.

First note that given any point $[\eta] \in \mathfrak{X}$, there is a unique sequence of elements 
\[
x_1^{\epsilon_1},x_2^{\epsilon_2},x_3^{\epsilon_3},\ldots
\]
 where $x_i \in \{f,h\}$ and $\epsilon_i \in \{ \pm 1\}$ with the property that $[\eta]$ is contained in the nested intersection
\[ \bigcap_{i=1}^\infty x_1^{\epsilon_1} \cdots x_i^{\epsilon_i}(B^{\epsilon_i}_{x_i}). \]
where $B^{\pm1}_h = B^\pm_h$ and $B^{\pm1}_f =B^\pm_f$.  
Identifying $\partial G$ with the set of infinite reduced words, our map is given there by
\[ 
\mathfrak{K}([\eta]) = x_1^{\epsilon_1}x_2^{\epsilon_2}x_3^{\epsilon_3}\cdots. 
\]

To see that $\mathfrak{K}$ is continuous, let $\mathcal U_g \subset \overline{\mathcal G}$ be the open set consisting of all infinite reduced words in $\partial G$ with prefix $g$ together with the union of the open tails of the corresponding paths in $\mathcal G$.  Now if $g$ ends in $x_0^{\epsilon_0}$ with $x_0 \in \{h,f\}$ and $\epsilon_0 \in \{\pm 1\}$, then
\[ 
\mathfrak{K}^{-1}(\mathcal U_g) =  gx_0^{-\epsilon_0}(\mathrm{int}\, B^{\epsilon_0}_{x_0})
\]
which is open.
\end{proof}

\begin{lemma}
$\mathfrak{K}$ is a one-sided inverse to $\mathfrak{I}$.  That is, $\mathfrak{K} \circ \mathfrak{I} = \mathrm{id}_{\partial G}$.
\end{lemma}
\begin{proof}
Since $\mathfrak{X}$ is a $G$--invariant closed set, it contains $\Lambda_G$, and so $\mathfrak{K} \circ \mathfrak{I}$ is well-defined.
Next, suppose that $x^+_f$ is the attracting fixed point of $f$.  Then $\mathfrak{I}(x^+_f) = [\mu^+_f]$ is the attracting fixed point in $\PML(S)$ of $f$, and hence $\mathfrak{K}(\mathfrak{I}(x^+_f)) = x^+_f$.  The same is true for any conjugate of $f$, and hence $\mathfrak{K} \circ \mathfrak{I}$ is the identity on the set of attracting fixed points of conjugates of $f$.  Being $G$--invariant, this set is dense in $\partial G$, and so, by continuity, $\mathfrak{K} \circ \mathfrak{I}$ is the identity on all of $\partial G$.
\end{proof}

Theorem \ref{pdexample2} follows easily from the following lemma.

\begin{lemma} \label{fiberszero}
For all $x \in \partial G$, we have $\mathfrak{K}^{-1}(x) \subset Z \mathfrak{I}(x)$.
In fact, if $x$ is the repelling fixed point $x^-_{^g\!h}$ of a conjugate $^gh$ of $h$, then $\mathfrak{K}^{-1}(x) = g Z \mu^-_h$.
Otherwise, the set $\mathfrak{K}^{-1}(x)$ is a singleton contained in $\Lambda_G$.
\end{lemma}

\begin{proof}[Proof of Theorem \ref{pdexample2} assuming Lemma \ref{fiberszero}]
By the first statement, $\mathfrak{X} \subset Z \Lambda_G$ since
\[ 
Z \Lambda_G = \bigcup_{x \in \partial G} Z \mathfrak{I}(x).
\]
So $\Omega \supset \Delta_G$ as required.
Again, the containment is proper as $Z\mu^+_h$ nontrivially intersects $\Omega$.

The description of $\Omega$ follows from the second and third statements.
\end{proof}

We need the following general fact about sequences of laminations.

\begin{lemma} \label{boundedangles}
Suppose $\mathfrak{S} \subset \ML(S)$ is a compact set, $\{f_k\} \subset \Mod(S)$ is an infinite sequence of distinct pseudo-Anosov mapping classes with 
\[
\mu^\pm_{f_k} \to \mu^\pm
\]
in $\ML(S)$, and that $\{\nu_k \}_{k=1}^\infty \subset \mathfrak{S}$ and $\{t_k \}_{k=1}^\infty \subset \mathbb R_+$ are sequences with 
\[
t_k f_k(\nu_k) \to \eta
\]
in $\ML(S)$.

If there is an $r > 0$ such that
\[ i(\nu,\mu^\pm) > r\]
for all $\nu \in \mathfrak{S}$, then $t_k \to 0$.
\end{lemma}
\begin{proof}
Note that continuity of $i$ and compactness of $\mathfrak{S}$ imply that there exist $K > 0$ and $R > 1$ such that for all $k \geq K$ and all $\nu \in \mathfrak{S}$
\[ 
\frac{1}{R} < i\big(\nu,\mu^\pm_{f_k}\big) < R .
\]

By the continuity of $i$ we have
\[ 
\lim_{k\to\infty} i\big(t_kf_k(\nu_k),\mu^-_{f_k}\big) = i(\eta,\mu^-),
\]
and so, for sufficiently large $k$, we have
\[ 
i(\eta,\mu^-) - 1 <  i\big(t_kf_k(\nu_k),\mu^-_{f_k}\big) < i(\eta,\mu^-) + 1. 
\]
The central term of this inequality is also given by
\begin{align*} i\big(t_kf_k(\nu_k),\mu^-_{f_k}\big) & =  i\big(t_k \nu_k,f_k^{-1}(\mu^-_{f_k})\big) \\
 & =  t_k i\big(\nu_k,\lambda(f_k^{-1})\mu^-_{f_k}\big) \\
 & =  t_k \lambda(f_k)i\big(\nu_k,\mu^-_{f_k}\big)
 \end{align*}
 where $\lambda(f_k)$ is the expansion factor of $f_k$, and so, for all sufficiently large $k$, we have
\[ 
\frac{i(\eta,\mu^-) - 1}{R} < t_k \lambda(f_k) < R(i(\eta,\mu^-) + 1).
\]
Since the $f_k$ are all distinct, and their fixed points converge in $\PML(S)$, it follows that $\lambda(f_k) \to \infty$.  So $t_k \to 0$ as required.
\end{proof}

\begin{proof}[Proof of Lemma \ref{fiberszero}]
First assume that $x \in \partial G$ is the fixed point of a conjugate of $h$.  
By the $G$--equivariance of $\mathfrak{K}$, it suffices to consider the case of $h$ itself.  
Then, we have $x = x^+_h$ or $x = x^-_h$.  In this case, the sequences of balls nesting to $\mathfrak{K}^{-1}(x^+_h)$ and $\mathfrak{K}^{-1}(x^-_h)$ are given by
\[ 
\{h^k(B^+_h)\}_{k=1}^\infty \mbox{ and } \{h^{-k}(B^-_h)\}_{k=1}^\infty, 
\]
respectively.

From the discussion in Section \ref{prelimsection}, we already know that
\[ 
\mathfrak{K}^{-1}(x^+_h) = \bigcap_{k=1}^\infty h^k(B^+_h) = \{ [\mu^+_h]\} \subset Z \mathfrak{I}(x^+_h)
\]
and
\[ 
\mathfrak{K}^{-1}(x^-_h) = \bigcap_{k=1}^\infty h^{-k}(B^-_h) = Z\mu^-_h = Z \mathfrak{I}(x^-_h).
\]

If $g \in G$ is any other element not conjugate to a power of $h$, then, by Property \ref{kluniform}, $g$ is pseudo-Anosov, and the dynamical properties of pseudo-Anosov mapping classes discussed in Section \ref{prelimsection} implies
\[ 
\mathfrak{K}^{-1}(x^\pm(g)) = \{ [\mu^\pm(g)] \} = Z \mathfrak{I}(x^\pm(g)).
\]

Therefore, to complete the proof of the lemma, we assume that $x \in \partial G$ is not a fixed point of any element of $G$.

We write $x$ as an infinite reduced word
\[ 
x = x_1^{\epsilon_1}x_2^{\epsilon_2}x_3^{\epsilon_3}\cdots. 
\]
Since $x$ is not the fixed point of any element of $G$, we can assume that $x_n = f$ and, say, $\epsilon_n = +1$ for infinitely many $n$ (the case that $x_n = f$ and $\epsilon_n = -1$ for infinitely many $n$ is similar).  
The $G$--equivariance of $\mathfrak{K}$ implies that we may also assume that $x_1 = f$ and $\epsilon_1 = 1$.  
Let $\{n_k\}_{k=1}^\infty$ be the increasing sequence of natural numbers for which $x_{n_k} = f$ and $\epsilon_{n_k} = +1$.  
Finally, set
\[ 
f_k = x_1^{\epsilon_1} x_2^{\epsilon_2}x_3^{\epsilon_3} \cdots x_{n_k}^{\epsilon_{n_k}} \in G. 
\]

Then, we have $\mathfrak{K}^{-1}(x)$ expressed as the nested intersection
\[ 
\mathfrak{K}^{-1}(x) = \bigcap_{k=1}^\infty f_k(B^+_f).
\]
Any point $[\eta]$ in the frontier of $\mathfrak{K}^{-1}(x)$ is a limit of a sequence in the frontiers
\[ 
[\eta] = \lim_{k \to \infty} f_k([\nu_k])
\]
where
\[ [\nu_k] \in \mathrm{Fr}(B^+_f) = \Sigma^+_f.\]
We fix any such $[\eta] \in \mathrm{Fr}(\mathfrak{K}^{-1}(x))$ and such a sequence $\{[\nu_k]\}$.

We pass to a further subsequence so that $\mu^\pm_{f_k} \to \mu^\pm \in \ML(S)$.  Since $[\mu^\pm_{f_k}] \in \Lambda_G$ for all $k$, we also have $[\mu^\pm] \in \Lambda_G$.  In fact, since $f_k= x_1^{\epsilon_1}\cdots x_{n_k}^{\epsilon_{n_k}}$ is cyclically reduced, the axes for $f_k$ in $\mathcal G$ all go through the origin and limit to a geodesic $\gamma \subset \mathcal G$ through ${\bf 1}$ with positive ray ending at $x$.  Therefore, $x^+_{f_k} \to x$ as $k \to \infty$, and by continuity of $\mathfrak{I}$, it follows that
\[ 
\mathfrak{I}(x) = [\mu^+] \in \Lambda_G.
\]
Moreover, the negative ray of $\gamma$ ends at some point $y \in \partial G$ and is described by an infinite word
\[ 
y = y_1^{\delta_1} y_2^{\delta_2} y_3^{\delta_3} \cdots 
\]
where $y_1^{\delta_1} \neq f$ since $x_1^{\epsilon_1} =f$ and $\gamma$ is a geodesic.  Therefore, again appealing to the continuity of $\mathfrak{I}$ we see that
\[ 
\mathfrak{I}(y) = [\mu^-] \in \Lambda_G \cap \PML(S) - B^+_f.
\]

By similar reasoning, for any $[\mu] \in \Lambda_G \cap B^+_f$, we have
\[ 
f_k([\mu]) \to [\mu^+] = \mathfrak{I}(x).
\]
In fact, it follows from \cite[Lemma 2.7]{MP} that there is a $\mu$ (a fixed point of a pseudo-Anosov in $G$)
and a sequence $s_k$ tending to zero such that
\[ 
\lim_{k \to \infty} s_k f_k(\mu) = \mu^+ \in \ML(S).
\]

We now let $\mathfrak{S} \subset \ML(S)$ be the image of $\Sigma^+_f$ under some continuous section of $\ML(S) \to \PML(S)$.  
Since $\Sigma^+_f \cap Z \Lambda_G = \emptyset$, there is an $r > 0$ such that
\[
i(\nu,\mu^\pm) > r 
\]
for every $\nu \in \mathfrak{S}$.

We take the representatives $\nu_k$ of $[\nu_k]$ to lie in $\mathfrak{S}$.  
Then, according to Lemma \ref{boundedangles}, the sequence $t_k$ for which
\[
\lim_{k\to\infty} t_k f_k(\nu_k) = \eta
\]
must tend to zero.  
So
\[
i(\eta,\mu^+) = \lim_{k \to \infty} i\big(t_k f_k(\nu_k),s_k f_k(\mu)\big) = \lim_{k \to \infty} t_k s_k i(\nu_k,\mu) = 0
\]
since $s_k$ and $t_k$ tend to zero and $i(\nu_k,\mu)$ is uniformly bounded by compactness of $\mathfrak{S}$.
Since $\mu^+$ is uniquely ergodic, we conclude that $[\eta] = [\mu^+] = \mathfrak{I}(x)$.

This means that the frontier of $\mathfrak{K}^{-1}(x)$ is precisely $\{\mathfrak{I}(x)\}$, and hence
\[ \mathfrak{K}^{-1}(x) = \{\mathfrak{I}(x)\} = Z \mathfrak{I}(x)\]
as required.
\end{proof}

\section{Final comments}\label{finalcomments}

If we replace $h$ with $h^{-1}$ in our construction we obtain another $G$--invariant open set $\Omega'$ on which $G$ acts properly discontinuously and cocompactly.  
By Lemma \ref{fiberszero}, we have descriptions
\[
\Omega = \PML(S) - \left( \Lambda_G \cup \bigcup_{g\in G} g Z\mu^-_h \right)
\]
and
\[
\Omega' = \PML(S) - \left( \Lambda_G \cup \bigcup_{g\in G} g Z\mu^+_h \right),
\]
and it follows that
\[
\Omega \cup \Omega' = \PML(S)- \left(\Lambda_G \cup \bigcup_{g\in G} G \cdot \alpha\right).
\]  
The group $G$ does not act properly discontinuously on $\Omega \cup \Omega'$, and in fact, we have the following.

\begin{proposition} \label{maximal domains}
If $\mathcal{U} \subset \PML(S)$ is any open set on which $G$ acts properly discontinuously, then $\mathcal{U} \subset \Omega$ or $\mathcal{U} \subset \Omega'$.
\end{proposition}
\begin{proof}
Let $\mathcal{U} \subset \PML(S)$ be a $G$--invariant open set.  We will show that if $\mathcal{U}$ is not contained in either $\Omega$ or $\Omega'$, then $G$ does not act properly discontinuously on $\mathcal{U}$.

If $\mathcal{U} \cap \Lambda_G \neq \emptyset$, then since $G$ acts minimally on $\Lambda_G$ and $\mathcal{U}$ is $G$--invariant, we must have $\Lambda_G \subset \mathcal{U}$.  As $G$ clearly fails to act properly discontinuously on $\mathcal{U}$ in this case, we assume that $\mathcal{U} \cap \Lambda_G = \emptyset$.

So if $\mathcal{U}$ fails to be contained in either $\Omega$ or $\Omega'$, there are points $[\eta^+] \in \mathcal{U} \cap Z \mu^+_h$ and $[\eta^-] \in \mathcal{U} \cap Z \mu^-_h$.   Moreover, $[\eta^\pm]$ is in the interior of $Z \mu^\pm_h$.  Let $\Upsilon^\pm$ be small compact balls contained in $\mathcal{U}$ containing $[\eta^\pm]$.  Since $[\eta^+] \in \Omega$, we may assume that $\Upsilon^+ \subset \Omega$.  Moreover, $G$--invariance of $\mathcal{U}$ allows us to pick $[\eta^+]$ and $\Upsilon^+$ to lie in $B^-_h$.

After passing to a subsequence, we can assume that the sequence of sets \linebreak $\{h^{-{k_j}}(\Upsilon^+)\}_{j=1}^\infty$ converges in the Hausdorff topology.
Moreover, we have
\[ \lim_{j \to \infty}h^{-k_j}(\Upsilon^+) \subset \bigcap_{k=1}^\infty h^{-k}(B^-_k) = Z \mu^-_h.\]
Note that the Hausdorff limit must be connected since $\Upsilon^+$ is.  
This limit contains $\alpha$ as the pointwise limit of $h^{-k}[\eta^+]$, and $[\mu^-_h]$ as the pointwise limit of any other point of $\Upsilon^+$ under $h^{-k}$.  
Therefore,
\[ \lim_{j \to \infty}h^{-k_j}(\Upsilon^+) = Z \mu^-_h.\]

Now, consider the compact set $\Upsilon = \Upsilon^+ \cup \Upsilon^-$.  Since $\mathrm{int}(\Upsilon^-)$ is a neighborhood of $[\eta^-]$, we have
\[
h^{-k_j}(\Upsilon) \cap \Upsilon \supset h^{-k_j}(\Upsilon^+) \cap \mathrm{int}(\Upsilon^-) \neq \emptyset
\]
for all sufficiently large $j$.
So $G$ does not act properly discontinuously on $\mathcal{U}$.
\end{proof}

From this we deduce that $\Omega$ and $\Omega'$ are the only maximal open sets on which $G$ acts properly discontinuously.  
By our descriptions of $\Omega$ and $\Omega'$ we also have
\[ \Delta_G = \Omega \cap \Omega'.\]
It follows that $\Delta_G$ can be described purely in terms of the action of $G$ on $\PML(S)$, \textit{without referring to geometric structures on the surface}.

Though $\Delta_G$ may not be a \textit{maximal} open set on which $G$ acts nicely, it remains a \textit{canonically defined} one, and we pose the following question.

\begin{question}
If $G$ is an irreducible subgroup of $\Mod(S)$, is $\Delta_G$ is the intersection of all maximal open sets on which $G$ acts properly discontinuously?
\end{question}

\bibliographystyle{amsalpha}

\end{document}